\documentclass[11pt]{article}

\usepackage[a4paper,margin=1in]{geometry}
\usepackage{amsmath,amssymb,amsthm,mathtools}
\usepackage{enumitem}
\usepackage{microtype}
\usepackage[colorlinks=true,citecolor=blue,linkcolor=blue,urlcolor=blue]{hyperref}

\allowdisplaybreaks[3]
\AtBeginDocument{%
  \setlength{\abovedisplayskip}{6pt plus 2pt minus 3pt}%
  \setlength{\belowdisplayskip}{6pt plus 2pt minus 3pt}%
  \setlength{\abovedisplayshortskip}{4pt plus 2pt minus 2pt}%
  \setlength{\belowdisplayshortskip}{4pt plus 2pt minus 2pt}%
}

\makeatletter
\renewcommand\section{\@startsection{section}{1}{\z@}%
  {-1.8ex plus -0.5ex minus -0.2ex}%
  {0.8ex plus 0.2ex}%
  {\normalfont\large\bfseries}}

\renewcommand\subsection{\@startsection{subsection}{2}{\z@}%
  {-1.4ex plus -0.4ex minus -0.2ex}%
  {0.6ex plus 0.2ex}%
  {\normalfont\normalsize\bfseries}}

\renewcommand\subsubsection{\@startsection{subsubsection}{3}{\z@}%
  {-1.0ex plus -0.3ex minus -0.2ex}%
  {0.5ex plus 0.2ex}%
  {\normalfont\normalsize\itshape}}
\makeatother

\newtheorem{theorem}{Theorem}[section]
\newtheorem{proposition}[theorem]{Proposition}
\newtheorem{lemma}[theorem]{Lemma}
\newtheorem{corollary}[theorem]{Corollary}

\theoremstyle{definition}
\newtheorem{definition}[theorem]{Definition}

\theoremstyle{remark}
\newtheorem{remark}[theorem]{Remark}

\title{Extremal Deletion-Ball Intersections under Run-Count and Lower-Order Deletion-Ball Intersection Constraints\thanks{This research was supported by the National Key Research and Development Program of China under Grant 2025YFC3409900, the National Natural Science Foundation of China under Grant 12231014, and Beijing Scholars Program.}}

\author{
Yubo Sun\thanks{
Institute of Mathematics and Interdisciplinary Sciences,
Xidian University, Xi'an 710126, China.
Email: \texttt{ybsun@cnu.edu.cn}.
}
\and
Gennian Ge\thanks{
Corresponding author.
School of Mathematical Sciences,
Capital Normal University, Beijing 100048, China.
Email: \texttt{gnge@zju.edu.cn}.
}
}

\date{}

\begin{document}

\maketitle

\begin{abstract}
Motivated by sequence reconstruction and reconstruction codes, we study extremal intersections of deletion balls over a fixed $q$-ary alphabet.
Let $\Sigma_q^n$ be the set of sequences of length $n$ over
$\Sigma_q$, and let $D_t(x)$ denote the set of all sequences obtained from $x\in\Sigma_q^n$ by deleting exactly $t$ symbols.

Our first result gives a finite upper bound under a lower-order deletion-correction constraint. We prove that if
$x,y\in\Sigma_q^n$ satisfy $D_{s-1}(x)\cap D_{s-1}(y)=\varnothing$, then
\[
|D_t(x)\cap D_t(y)|
\le
\binom{2s}{s}\binom{n-s}{t-s}.
\]
For binary alphabets, this strengthens a recent asymptotic upper bound of Pham, Goyal, and Kiah (2025, JCTA).
We then investigate deletion-ball intersections under simultaneous constraints on run counts and lower-order deletion-ball intersections.
For fixed $0<\gamma\le1$, integers $1\le s\le t$, and $m\ge1$, we show that if $x,y\in\Sigma_q^n$ have at most $\gamma n$ runs and satisfy $|D_s(x)\cap D_s(y)|\le m$, then
\[
|D_t(x)\cap D_t(y)|\le
\frac{m\gamma^{t-s}}{(t-s)!}n^{t-s}+O_{s,t,m}(n^{t-s-1}).
\]
Moreover, the leading term can be attainable whenever $m$ is realized by a fixed finite-length seed pair. 
As a consequence, we obtain a direct lifting theorem for deletion reconstruction codes, transferring reconstruction properties from radius $s$ to larger radii $t$.
Finally, we establish a parallel insertion theory and derive corresponding results for insertion-ball intersections and insertion
reconstruction codes.
\end{abstract}

\noindent\textbf{Keywords.}
Sequence reconstruction; reconstruction code; deletion ball; common subsequence

\medskip

\section{Introduction}

The study of intersections of error balls dates back to Levenshtein's foundational work on the sequence reconstruction problem
\cite{Levenshtein-01-IT,Levenshtein-01-JCTA}. In the reconstruction setting, a codeword is transmitted through several identical noisy channels, and the receiver attempts to recover the original codeword from a collection of distinct channel outputs. The maximum intersection size of error balls is therefore a central quantity: it determines the reconstruction threshold, namely, the minimum number of distinct noisy
reads that guarantees unique reconstruction. Motivated by emerging storage technologies such as DNA storage \cite{Church-12-science,Goldman-13-Nature} and racetrack memory \cite{Chee-18-IT}, which naturally provide multiple noisy reads, the
extremal behavior of deletion-ball and insertion-ball intersections has
attracted considerable attention in recent years \cite{Gabrys-18-IT,Chrisnata-22-IT,Sala-17-IT,Sun-26-IT,Ye-23-IT,Pham-25-JCTA,Wang-26-IT,Wang-26-arXiv}.

Let $\Sigma_q$ be an alphabet of size $q\ge 2$, and let $\Sigma_q^n$ denote the set of all sequences of length $n$ over $\Sigma_q$. For two sequences $x\in\Sigma_q^n$ and $y\in\Sigma_q^m$, we write $y\preccurlyeq x$ if $y$ is a \emph{subsequence} of $x$, that is, if $y$ can be obtained from $x$ by deleting exactly $n-m$ symbols. In this
case, $x$ is called a \emph{supersequence} of $y$.

For an integer $t\ge 0$, the \emph{$t$-deletion ball} of $x\in\Sigma_q^n$ is defined as
\[
D_t(x)=\{w\in\Sigma_q^{n-t}: w\preccurlyeq x\}.
\]
Thus $D_t(x)$ consists of all distinct sequences obtained from $x$ by deleting exactly $t$ symbols. Deletion balls are irregular: their sizes depend on the structure of the center sequence. Let
\[
D_q(n,t)=\max_{x\in\Sigma_q^n}|D_t(x)|
\]
be the largest size of a $t$-deletion ball centered at a $q$-ary sequence of length $n$. Calabi \cite{Calabi1967} and later Hirschberg and R{\'e}gnier \cite{Hirschberg} determined that
\[
D_q(n,t)=\sum_{i=0}^t \binom{n-t}{i}D_{q-1}(t,t-i).
\]

The \emph{$t$-insertion ball} of $x\in\Sigma_q^n$ is defined as
\[
I_t(x)=\{w\in\Sigma_q^{n+t}: x\preccurlyeq w\},
\]
namely, the set of all distinct sequences obtained from $x$ by inserting exactly $t$ symbols. In contrast to deletion balls, insertion balls are regular: their sizes depend only on $n,q$, and $t$, and not on the particular center. Levenshtein \cite{Levenshtein-66} proved that, for
every $x\in\Sigma_q^n$,
\[
I_q(n,t):=|I_t(x)|
=
\sum_{i=0}^t \binom{n+t}{i}(q-1)^i.
\]

For fixed positive integer $t$, a basic extremal problem is to determine the largest
possible values of
\[
|D_t(x)\cap D_t(y)|
\qquad\text{and}\qquad
|I_t(x)\cap I_t(y)|
\]
under suitable restrictions on the pair $x,y$.

\subsection{Previous Work}

In his seminal work \cite{Levenshtein-01-JCTA}, Levenshtein determined these quantities in the unconstrained case:
\begin{align*}
\max_{\substack{x,y\in\Sigma_q^n\\x\ne y}}
|D_t(x)\cap D_t(y)|
&=
D_q(n,t)-D_q(n-1,t)+D_q(n-2,t-1)  \\
&=
\frac{2}{(t-1)!}n^{t-1}+O_{q,t}(n^{t-2}),
\end{align*}
\[
\max_{\substack{x,y\in\Sigma_q^n\\x\ne y}}
|I_t(x)\cap I_t(y)|
=
2\sum_{i=1}^t I_q(n,t-i)(q-2)^i
=
\frac{2(q-1)^{t-1}}{(t-1)!}n^{t-1}
+
O_{q,t}(n^{t-2}).
\]
Thus, for arbitrary distinct centers, the largest intersection is of order $n^{t-1}$.

A natural way to reduce the intersection between two error balls is to impose a error-correction separation condition on the two centers. For sequences of the same length, $(s-1)$-deletion correction is equivalent to
$(s-1)$-insertion correction; that is,
\[
D_{s-1}(x)\cap D_{s-1}(y)=\varnothing
\qquad\Longleftrightarrow\qquad
I_{s-1}(x)\cap I_{s-1}(y)=\varnothing.
\]
Such a condition arises naturally in synchronization-error correction and
DNA storage. Here $s\ge 1$, and the case $s=1$ reduces exactly to the
unconstrained condition $x\ne y$. Under this separation condition, the
order of the higher-order intersection drops from $n^{t-1}$ in the unrestricted case to $n^{t-s}$.

Let $t\ge s\ge 2$ be fixed. For insertion balls, Sala, Gabrys, Schoeny, and Dolecek~\
\cite{Sala-17-IT} proved that
\[
\max_{\substack{x,y\in\Sigma_q^n\\I_{s-1}(x)\cap I_{s-1}(y)=\varnothing}}
|I_t(x)\cap I_t(y)|=
\frac{(q-1)^{t-s}\binom{2s}{s}}{(t-s)!}
n^{t-s} + O_{q,s,t}(n^{t-s-1}).
\]
For deletion balls, the problem is more delicate, due in part to the irregularity of deletion balls. In the binary case with $s=2$, Gabrys and Yaakobi \cite{Gabrys-18-IT} showed that
\[
\max_{\substack{x,y\in\Sigma_2^n\\D_1(x)\cap D_1(y)=\varnothing}}
|D_t(x)\cap D_t(y)|
=
\frac{6}{(t-2)!}n^{t-2}+O_t(n^{t-3}).
\]
More recently, Wang, Li, and Fu \cite{Wang-26-IT} extended this result to ternary alphabets. For general $s\ge 2$, Pham, Goyal, and Kiah~\cite{Pham-25-JCTA} asymptotically solved the binary case by proving
\[
\max_{\substack{x,y\in\Sigma_2^n\\
D_{s-1}(x)\cap D_{s-1}(y)=\varnothing}}
|D_t(x)\cap D_t(y)|
=
\frac{\binom{2s}{s}}{(t-s)!}n^{t-s}
+
O_{s,t}(n^{t-s-1}).
\]
For general alphabets, Sun and Ge \cite{Sun-26-IT} recently obtained an upper bound by constructing an injection from a deletion-ball intersection into the corresponding insertion-ball intersection.
Combining this injection with known results on insertion-ball intersections, they showed that
\[
\max_{\substack{x,y\in\Sigma_q^n\\
D_{s-1}(x)\cap D_{s-1}(y)=\varnothing}}
|D_t(x)\cap D_t(y)|
\le
\frac{(q-1)^{t-s}\binom{2s}{s}}{(t-s)!}
n^{t-s}
+
O_{q,s,t}(n^{t-s-1}).
\]
However, this bound is not tight in general, even for $q=3$ \cite{Wang-26-IT}. Obtaining the correct leading term for non-binary
deletion-ball intersections therefore remains open.

\subsection{Overview and Main Results}

The goal of this paper is to address the above leading-term problem and, more generally, to study deletion-ball intersections in a refined setting. Rather than assuming only that a lower-order intersection is empty, we ask how an upper bound on a lower-order intersection controls higher-order intersections. More precisely, suppose that
\[
|D_s(x)\cap D_s(y)|\le m
\]
for some fixed integer $m$. How large can $|D_t(x)\cap D_t(y)|$ be for
$t\ge s$? The special case $s=1$ and $m=1$ was studied in \cite{Chrisnata-22-IT}, and the corresponding insertion analogue was
considered in \cite{Ye-23-IT,Sun-26-IT}. For general parameters, however, this problem remains open.

We also incorporate a run-count constraint. Let $r(x)$ denote the number of runs of $x$. Such constraints are natural in deletion channels because the size of a deletion ball is governed by the number of runs of its
center. For fixed $0<\gamma\le1$ and integers $1\le s\le t$ and $m\geq 1$, one of our main results determines the leading term of 
\begin{equation*}
\max_{\substack{x,y\in\Sigma_q^n\\
r(x),\,r(y)\le \gamma n\\
|D_s(x)\cap D_s(y)|\le m}}
|D_t(x)\cap D_t(y)|.
\end{equation*}
The leading coefficient is governed by two transparent parameters: the number $m$ of common lower-order deletion descendants and the run-density parameter $\gamma$.

We now state our main results. The first result gives a finite upper bound under the classical lower-order deletion-correction condition.

\begin{theorem}\label{thm:finite-bound}
Let $x,y\in\Sigma_q^n$, and let $1\le s\le t\le n$. If
\[
D_{s-1}(x)\cap D_{s-1}(y)=\varnothing,
\]
then
\[
|D_t(x)\cap D_t(y)|
\le
\binom{2s}{s}\binom{n-s}{t-s}.
\]
In particular, if $s<t$ are fixed, then there exist constants $K>0$ and $n_0$, depending only on $s$ and $t$, such that for all $n\ge n_0$,
\[
|D_t(x)\cap D_t(y)|
\le
\frac{\binom{2s}{s}}{(t-s)!}n^{t-s}-K n^{t-s-1}.
\]
\end{theorem}

For binary alphabets, Theorem~\ref{thm:finite-bound} strengthens the asymptotic upper bound of Pham, Goyal, and Kiah \cite{Pham-25-JCTA}. Moreover, for general alphabets, it improves the upper bound of Sun and Ge \cite{Sun-26-IT} by a factor of $(q-1)^{t-s}$ in the leading term.

Our second result treats the setting in which both the run counts and the lower-order deletion-ball intersections are constrained.

\begin{theorem}\label{thm:run-upper}
Fix $0<\gamma\le 1$ and integers $q\ge 2$, $1\le s\le t$, and $m\geq 1$. As $n\to\infty$,
\[
\max_{\substack{x,y\in\Sigma_q^n\\
r(x),\,r(y)\le \gamma n\\
|D_s(x)\cap D_s(y)|\le m}}
|D_t(x)\cap D_t(y)|
\le
\frac{m\gamma^{t-s}}{(t-s)!}n^{t-s} +
O_{s,t,m}\!\left(n^{t-s-1}\right).
\]
\end{theorem}

Theorem~\ref{thm:run-upper} has a direct consequence for reconstruction codes.

\begin{definition}
An \emph{$(n,q,N;D_t)$-reconstruction code} is a code $\mathcal C\subseteq\Sigma_q^n$ such that $|D_t(x)\cap D_t(y)|<N$ for any two distinct codewords $x,y\in\mathcal C$.
\end{definition}

As an immediate consequence of the definition and Theorem~\ref{thm:run-upper}, we obtain the following lifting result.

\begin{corollary}\label{cor:deletion-reconstruction-lift}
Fix $0<\gamma\le 1$ and integers $q\ge 2$, $1\le s\le t$, and $m\geq 1$. Let $\mathcal C\subseteq\Sigma_q^n$ be an
$(n,q,m+1;D_s)$-reconstruction code satisfying $r(x)\le \gamma n$ for every $x\in\mathcal C$. Then $\mathcal C$ is also an $(n,q,N+1;D_t)$-reconstruction code, where, as $n\to\infty$,
\[
N
\le
\frac{m\gamma^{t-s}}{(t-s)!}n^{t-s} +
O_{s,t,m}\!\left(n^{t-s-1}\right).
\]
\end{corollary}

The upper bound in Theorem~\ref{thm:run-upper} is sharp in its leading term whenever the lower-order intersection size can be realized by a fixed finite-length seed pair.

\begin{definition}\label{def:feasible}
For fixed integers $q\ge2$ and $s\ge1$, a positive integer $m$ is called
\emph{$(q,s)$-deletion-feasible} if there exist two sequences $u,v\in\Sigma_q^L$, for some constant $L$, such that
\[
D_{s-1}(u)\cap D_{s-1}(v)=\varnothing \qquad \text{and} \qquad |D_s(u)\cap D_s(v)|=m.
\]
\end{definition}

\begin{remark}
  The assumption $D_{s-1}(x)\cap D_{s-1}(y)=\varnothing$ in Definition \ref{def:feasible} is natural. Without it, the level-$s$ intersection is no longer primitive: it already contains descendants arising from common lower-order deletion descendants. In this regime, the size of $D_s(x)\cap D_s(y)$ may grow with the block length. For instance, it can be of order $\Omega(n)$ when the underlying sequences have linearly many runs, a condition satisfied by almost all sequences. Hence such values of $m$ cannot, in general, be realized by a fixed finite-length seed pair in the sense required by Theorem~\ref{thm:main}. 

By Theorem~\ref{thm:finite-bound}, if $D_{s-1}(x)\cap D_{s-1}(y)=\varnothing$, then $|D_s(x)\cap D_s(y)|\leq \binom{2s}{s}$. This suggests that the finite feasible parameters most relevant to our asymptotic results are those satisfying $m\le \binom{2s}{s}$. Moreover, at the upper endpoint, Alon, Bourla, Graham, He, and Kravitz \cite{Alon-24-IT}, and
independently Pham, Goyal, and Kiah \cite{Pham-25-JCTA}, showed that $\binom{2s}{s}$ is $(q,s)$-deletion-feasible for every $q\ge2$. 
\end{remark}

\begin{theorem}\label{thm:main}
Fix $0<\gamma\le 1$ and integers $q\ge 2$ and $1\le s\le t$. Let $m$ be a
$(q,s)$-deletion-feasible integer, as $n\to\infty$,
\[
\max_{\substack{x,y\in\Sigma_q^n\\
r(x),\,r(y)\le \gamma n\\
|D_s(x)\cap D_s(y)|\le m}}
|D_t(x)\cap D_t(y)|
=
\frac{m\gamma^{t-s}}{(t-s)!}n^{t-s} + O_{s,t,m}\!\left(n^{t-s-1}\right).
\]
\end{theorem}

The feasibility assumption in Theorem~\ref{thm:main} is used only for
the lower bound. For example, if $s=1$ and $x\ne y$, then $|D_1(x)\cap D_1(y)|\le2$; hence a statement with leading coefficient
$m>2$ cannot hold without an additional realization assumption.

The proof of Theorems~\ref{thm:run-upper} is based on a layer-by-layer analysis of common subsequences. We first establish an upper bound on the number of maximal common
subsequences. We then introduce a primitive-layer decomposition, which separates $D_t(x)\cap D_t(y)$ according to the first level at which a common subsequence becomes primitive, or equivalently maximal with respect to extension. The layers above level $s$ contribute only lower-order terms, while the prescribed intersection $D_s(x)\cap D_s(y)$ gives the main term. Combining this decomposition with classical run-count bounds for deletion balls yields the upper
bound. The matching lower bound in Theorem~\ref{thm:main} is obtained by appending a long tail with controlled run count to a fixed finite-length seed pair.

We also develop the insertion analogue of this theory.

\begin{theorem}\label{thm:finite-insertion-bound}
Fix integers $q\ge2$, $1\le s\le t$, and $m\geq 1$. As $n\to\infty$, 
\[
\max_{\substack{x,y\in\Sigma_q^n\\
|I_s(x)\cap I_s(y)|\le m}}
|I_t(x)\cap I_t(y)|
\le
\frac{m(q-1)^{t-s}}{(t-s)!}n^{t-s} + O_{q,s,t,m}\!\left(n^{t-s-1}\right).
\]
\end{theorem}

Similarly, Theorem~\ref{thm:finite-insertion-bound} yields a lifting
result for insertion reconstruction codes.

\begin{definition}
An \emph{$(n,q,N;I_t)$-reconstruction code} is a code $\mathcal C\subseteq\Sigma_q^n$ such that $|I_t(x)\cap I_t(y)|<N$ for any two distinct codewords $x,y\in\mathcal C$.
\end{definition}

\begin{corollary}\label{cor:insertion-reconstruction-lift}
Fix integers $q\ge2$, $1\le s\le t$, and $m\ge1$. Any $(n,q,m+1;I_s)$-reconstruction code is also an $(n,q,N+1;I_t)$-reconstruction code, where, as $n\to\infty$,
\[
N
\le
\frac{m(q-1)^{t-s}}{(t-s)!}n^{t-s} +
O_{q,s,t,m}\!\left(n^{t-s-1}\right).
\]
\end{corollary}

The leading term in Theorem~\ref{thm:finite-insertion-bound} is sharp whenever the parameter $m$ can be realized by a fixed finite-length seed pair.

\begin{definition}\label{def:insertion-feasible}
For fixed integers $q\ge2$ and $s\ge1$, a positive integer $m$ is called
\emph{$(q,s)$-insertion-feasible} if there exist two sequences $u,v\in\Sigma_q^L$, for some constant $L$, such that
\[
I_{s-1}(u)\cap I_{s-1}(v)=\varnothing \qquad \text{and} \qquad |I_s(u)\cap I_s(v)|=m.
\]
\end{definition}

As in the deletion case, every $(q,s)$-insertion-feasible integer $m$ satisfies $m\leq \binom{2s}{s}$ and the assumption
$I_{s-1}(x)\cap I_{s-1}(y)=\varnothing$ is natural. Moreover, at the upper endpoint, Sala, Gabrys, Schoeny, and Dolecek \cite{Sala-17-IT} showed that $\binom{2s}{s}$ is $(q,s)$-insertion-feasible for every $q\ge2$.

\begin{theorem}\label{thm:insertion-main}
Fix integers $q\ge2$ and $1\le s\le t$. Let $m$ be a $(q,s)$-insertion-feasible integer, as $n\to\infty$,
\[
\max_{\substack{x,y\in\Sigma_q^n\\
|I_s(x)\cap I_s(y)|\le m}}
|I_t(x)\cap I_t(y)|
=
\frac{m(q-1)^{t-s}}{(t-s)!}n^{t-s}
+
O_{q,s,t,m}\!\left(n^{t-s-1}\right).
\]
\end{theorem}

As in the deletion case, the feasibility assumption in Theorem~\ref{thm:insertion-main} is needed only for the lower bound.


\subsection{Organization}

The remainder of the paper is organized as follows.
Section~\ref{sec:finite} proves Theorem~\ref{thm:finite-bound},
establishing the finite upper bound under lower-order
deletion-correction constraints. Section~\ref{sec:layers} introduces
maximal common subsequences and the primitive-layer decomposition, and
uses them to prove Theorem~\ref{thm:run-upper}. Section~\ref{sec:lower}
establishes the matching lower bound and proves Theorem~\ref{thm:main}.
Section~\ref{sec:insertion} develops the insertion analogue via minimal
common supersequences and proves
Theorems~\ref{thm:finite-insertion-bound} and~\ref{thm:insertion-main}.
Finally, Section~\ref{sec:conclusion} concludes the paper and highlights
several directions for future research.

\section{Proof of Theorem \ref{thm:finite-bound}}\label{sec:finite}

We begin by defining common subsequences of fixed length.

\begin{definition}
For two sequences $u,v$ and an integer $r\ge 0$, let
\[
\mathcal C_r(u,v)
:=
\{w:\ |w|=r,\ w\preccurlyeq u,\ w\preccurlyeq v\}
\]
denote the set of all common subsequences of $u$ and $v$ of length $r$.
The length of a longest common subsequence of $u$ and $v$ is denoted by
$\operatorname{LCS}(u,v)$; equivalently,
\[
\operatorname{LCS}(u,v)
=
\max\{r:\mathcal C_r(u,v)\ne\varnothing\}.
\]
\end{definition}

\begin{remark}
For two sequences $x,y\in\Sigma_q^n$ and an integer $1\le s\le n$,
\[
D_{s-1}(x)\cap D_{s-1}(y)=\varnothing
\qquad\Longleftrightarrow\qquad
\operatorname{LCS}(x,y)\le n-s.
\]
\end{remark}

\begin{lemma}\label{lem:barrier-shadow}
Let $a,b,N$ be nonnegative integers, and let $u,v$ be sequences satisfying
\[
|u|=N+a,
\qquad
|v|=N+b,
\qquad
\operatorname{LCS}(u,v)\le N.
\]
Then, for every $0\le r\le N$,
\[
|\mathcal C_r(u,v)|\le \binom{a+b}{a}\binom Nr .
\]
\end{lemma}

\begin{proof}
We argue by induction on $|u|+|v|$. The cases $r=0$, $a=0$, or $b=0$
are immediate. Assume $r,a,b\ge 1$, and write
\[
u=\alpha u',\qquad v=\beta v',
\]
where $\alpha$ and $\beta$ are the first symbols of $u$ and $v$, respectively.

First suppose that $\alpha=\beta$. Then
$\operatorname{LCS}(u',v')\le N-1$; otherwise the common first symbol $\alpha$
would extend a common subsequence of $u'$ and $v'$ of length $N$ to one
of $u$ and $v$ of length $N+1$. Partition $\mathcal C_r(u,v)$ according
to whether its elements start with $\alpha$. If $w=\alpha w'$, then
$w'\in\mathcal C_{r-1}(u',v')$; if $w$ does not start with $\alpha$, then
$w\in\mathcal C_r(u',v')$. Hence
\[
|\mathcal C_r(u,v)|
\le
|\mathcal C_{r-1}(u',v')|+|\mathcal C_r(u',v')|.
\]
Applying the induction hypothesis gives
\[
|\mathcal C_r(u,v)|
\le
\binom{a+b}{a}\binom{N-1}{r-1}
+
\binom{a+b}{a}\binom{N-1}{r}
=
\binom{a+b}{a}\binom Nr .
\]

It remains to consider the case $\alpha\ne\beta$. Partition
$\mathcal C_r(u,v)$ according to whether its elements start with $\alpha$.
If $w$ does not start with $\alpha$, then $w\in\mathcal C_r(u',v)$. If
$w$ starts with $\alpha$, then it cannot start with $\beta$, and hence
$w\in\mathcal C_r(u,v')$. 
This implies that
\[
|\mathcal C_r(u,v)|
\le
|\mathcal C_{r}(u',v)|+|\mathcal C_r(u,v')|.
\]
Moreover,
\[
\operatorname{LCS}(u',v)\le N,
\qquad
\operatorname{LCS}(u,v')\le N.
\]
Therefore, by the induction hypothesis,
\[
|\mathcal C_r(u,v)|
\le
\binom{a+b-1}{a-1}\binom Nr
+
\binom{a+b-1}{a}\binom Nr
=
\binom{a+b}{a}\binom Nr .
\]
This completes the proof.
\end{proof}

\begin{remark}
The condition $\operatorname{LCS}(u,v)\le N$ in
Lemma~\ref{lem:barrier-shadow} is essential. If $u=v$ has length $N+1$
and all symbols are distinct, then $|\mathcal C_N(u,v)|=N+1$, whereas
$\binom21\binom NN=2$.
\end{remark}

\begin{proof}[Proof of Theorem~\ref{thm:finite-bound}]
The assumption is equivalent to $\operatorname{LCS}(x,y)\le n-s$. Applying
Lemma~\ref{lem:barrier-shadow} with
\[
u=x,
\qquad
v=y,
\qquad
N=n-s,
\qquad
a=b=s,
\qquad
r=n-t,
\]
we obtain
\[
|D_t(x)\cap D_t(y)|
=|\mathcal C_{n-t}(x,y)|
\le
\binom{2s}{s}\binom{n-s}{n-t}
=
\binom{2s}{s}\binom{n-s}{t-s}.
\]
If $d=t-s\ge 1$, then
\[
\binom{n-s}{d}
=
\frac{n^d}{d!}
-
\frac{2s+(d-1)}{2(d-1)!}n^{d-1}
+O_{s,d}(n^{d-2}).
\]
This gives the asserted estimate with a negative second-order term.
\end{proof}

\section{Proof of Theorem \ref{thm:run-upper}}\label{sec:layers}

We begin by defining maximal common subsequences of fixed length.

\begin{definition}\label{def:maximal}
Let $u$ and $v$ be sequences and let $N\ge0$. A sequence $w$ is a
\emph{maximal common subsequence} of $u$ and $v$ of length $N$ if
$w\preccurlyeq u$, $w\preccurlyeq v$, and there is no common subsequence
$w^+$ of $u$ and $v$ of length $N+1$ such that $w\preccurlyeq w^+$.
Let $\mathcal M_N(u,v)$ be the set of all such distinct sequences.
\end{definition}

\begin{remark}
    A maximal common subsequence may not be a longest common subsequence.
\end{remark}

\begin{lemma}\label{lem:maximal-common-subsequences}
Let $a,b,N$ be nonnegative integers, and let $u,v$ be sequences satisfying
\[
|u|=N+a,
\qquad
|v|=N+b.
\]
Then
\[
|\mathcal M_N(u,v)|\le \binom{a+b}{a}.
\]
\end{lemma}

\begin{proof}
We argue by induction on $|u|+|v|$. The cases $N=0$, $a=0$, or $b=0$
are immediate. Assume $N,a,b\ge 1$, and write
\[
u=\alpha u',\qquad v=\beta v',
\]
where $\alpha$ and $\beta$ are the first symbols of $u$ and $v$, respectively.

First suppose that $\alpha=\beta$. If $w\in\mathcal M_N(u,v)$ does not
start with $\alpha$, then $w\preccurlyeq u'$ and $w\preccurlyeq v'$, so
$\alpha w$ would be a common subsequence of $u$ and $v$ of length $N+1$
containing $w$, a contradiction. Thus every $w\in\mathcal M_N(u,v)$ has
the form $w=\alpha w'$. Moreover, $w'$ must be a maximal common
subsequence of $u'$ and $v'$ of length $N-1$; otherwise a longer common
subsequence of $u'$ and $v'$ containing $w'$ would yield a common
subsequence of $u$ and $v$ of length $N+1$ containing $w$. 
By the induction hypothesis, we can derive
\[
|\mathcal M_N(u,v)|
\le
|\mathcal M_{N-1}(u',v')|
\le
\binom{a+b}{a}.
\]

It remains to consider the case $\alpha\ne\beta$. Partition
$\mathcal M_N(u,v)$ according to whether its elements start with
$\alpha$. If $w$ does not start with $\alpha$, then
$w\in\mathcal M_N(u',v)$. If $w$ starts with $\alpha$, then it cannot
start with $\beta$, and hence $w\in\mathcal M_N(u,v')$. By the induction hypothesis, we can derive
\[
|\mathcal M_N(u,v)|\le |\mathcal M_N(u',v)|+|\mathcal M_N(u,v')|
\le
\binom{a+b-1}{a-1}+\binom{a+b-1}{a}
=
\binom{a+b}{a}.
\]
This completes the proof.
\end{proof}

We now introduce the primitive-layer decomposition of common subsequences.

\begin{definition}\label{def:primitive-layer}
For $x,y\in\Sigma_q^n$ and $0\le j\le n$, define
\[
\mathcal L_j(x,y)
:=
D_j(x)\cap D_j(y)
\setminus
\bigcup_{z\in D_{j-1}(x)\cap D_{j-1}(y)}D_1(z),
\]
where $D_{-1}(x)\cap D_{-1}(y)=\varnothing$ by convention.
\end{definition}

\begin{lemma}[Primitive-layer decomposition]\label{lem:primitive-layer}
For any $x,y\in\Sigma_q^n$ and $0\le j\le n$,
\[
\mathcal L_j(x,y)=\mathcal M_{n-j}(x,y).
\]
Consequently,
\[
|\mathcal L_j(x,y)|\le \binom{2j}{j}.
\]
Moreover, if $1\le s\le t$, then
\[
D_t(x)\cap D_t(y)
=
\left( \bigcup_{z\in D_s(x)\cap D_s(y)}D_{t-s}(z) \right)
\cup \left(
\bigcup_{j=s+1}^{t}\ \bigcup_{z\in\mathcal L_j(x,y)}D_{t-j}(z) \right).
\]
\end{lemma}

\begin{proof}
A sequence $w\in D_j(x)\cap D_j(y)$ is not in $\mathcal L_j(x,y)$ if and
only if $w\in D_1(z)$ for some
$z\in D_{j-1}(x)\cap D_{j-1}(y)$. Equivalently, $w$ is contained in a
common subsequence of $x$ and $y$ of length $n-j+1$. Thus
$\mathcal L_j(x,y)$ is exactly the set of maximal common subsequences of
$x$ and $y$ of length $n-j$, namely
\[
\mathcal L_j(x,y)=\mathcal M_{n-j}(x,y).
\]
Applying Lemma~\ref{lem:maximal-common-subsequences} with
$N=n-j$ and $a=b=j$ gives
\[
|\mathcal L_j(x,y)|
=
|\mathcal M_{n-j}(x,y)|
\le
\binom{2j}{j}.
\]

It remains to prove the decomposition. 
The inclusion ``$\supseteq$'' is immediate.
For the reverse inclusion, take $w\in D_t(x)\cap D_t(y)$. If $w$ is contained in some sequence of $D_s(x)\cap D_s(y)$, then
$w$ belongs to the first union. Otherwise, choose the smallest $j>s$ for
which there exists $z\in D_j(x)\cap D_j(y)$ with $w\preccurlyeq z$. Such a $j$ exists
by taking $j=t$ and $z=w$. 
By the minimality of $j$, the chosen $z$ cannot be obtained by deleting one
symbol from a sequence in $D_{j-1}(x)\cap D_{j-1}(y)$, which means that $z\in\mathcal L_j(x,y)$. Since
$|z|=n-j$ and $|w|=n-t$, we have $w\in D_{t-j}(z)$. Hence $w$ belongs to the second union. This proves the reverse inclusion and completes the proof.
\end{proof}

The proof of Theorem~\ref{thm:run-upper} relies on the following classical
bound on the size of a deletion ball in terms of the run count of its
center.

\begin{lemma}[Levenshtein~\cite{Levenshtein-66}]
\label{lem:ballsize}
Let $z\in\Sigma_q^n$ be a sequence with $r(z)\le R$. Then
\[
\binom{R-k+1}{k}\leq |D_k(z)|
\le
\binom{R+k-1}{k}.
\]
\end{lemma}

\begin{proof}[Proof of Theorem \ref{thm:run-upper}]
Fix $x,y$ satisfying the constraints in Theorem \ref{thm:run-upper}, by the primitive-layer decomposition in Lemma~\ref{lem:primitive-layer},
\begin{equation}\label{eq:main-layer-cover}
D_t(x)\cap D_t(y)
=
\left( \bigcup_{z\in D_s(x)\cap D_s(y)}D_{t-s}(z) \right)
\cup \left(
\bigcup_{j=s+1}^{t}\ \bigcup_{z\in\mathcal L_j(x,y)}D_{t-j}(z) \right).
\end{equation}

We first estimate the contribution of the first union. For every
$z\in D_s(x)\cap D_s(y)$, since $z\preccurlyeq x$, deleting symbols cannot increase the
number of runs, and therefore
\[
r(z)\le r(x)\le \gamma n.
\]
By Lemma \ref{lem:ballsize},
\[
|D_{t-s}(z)|
\le
\binom{\lfloor\gamma n\rfloor+t-s-1}{t-s}
=
\frac{\gamma^{t-s}}{(t-s)!}n^{t-s}
+
O_{s,t}\!\left(n^{t-s-1}\right).
\]
Since $|D_s(x)\cap D_s(y)|\le m$, the first union in
\eqref{eq:main-layer-cover} contributes at most
\[
\frac{m\gamma^{t-s}}{(t-s)!}n^{t-s}
+
O_{s,t,m}\!\left(n^{t-s-1}\right).
\]

It remains to bound the second union. By Lemma~\ref{lem:primitive-layer},
\[
|\mathcal L_j(x,y)|\le \binom{2j}{j}.
\]
Together with the trivial estimate
\[
|D_{t-j}(z)|\le \binom{n-j}{t-j},
\]
we get
\[
\sum_{j=s+1}^{t}\sum_{z\in\mathcal L_j(x,y)}
|D_{t-j}(z)|
\le
\sum_{j=s+1}^{t}
\binom{2j}{j}\binom{n-j}{t-j}
=
O_{s,t}\!\left(n^{t-s-1}\right).
\]
Combining the two estimates gives
\[
|D_t(x)\cap D_t(y)|
\le
\frac{m\gamma^{t-s}}{(t-s)!}n^{t-s}
+
O_{s,t,m}\!\left(n^{t-s-1}\right),
\]
as claimed.
\end{proof}

\section{Proof of Theorem \ref{thm:main}}\label{sec:lower}

We begin with a common-suffix cancellation property, which is a $q$-ary generalization of \cite[Lemma~23]{Chrisnata-22-IT}.

\begin{lemma}\label{lem:suffix-cancellation}
Let $u,v\in\Sigma_q^L$ be such that
$D_{s-1}(u)\cap D_{s-1}(v)=\varnothing$.
Then, for every sequence $w$,
\[
D_s(uw)\cap D_s(vw)
=
\{cw:\ c\in D_s(u)\cap D_s(v)\}.
\]
\end{lemma}

\begin{proof}
It suffices to prove the claim when $w$ is a single symbol; the general
case then follows by iteration.

The inclusion ``$\supseteq$'' is immediate. For the reverse inclusion,
take $w\in\Sigma_q$ and $z\in D_s(uw)\cap D_s(vw)$. If $z$ does not end
with $w$, then the final symbol $w$ must have been deleted from both
$uw$ and $vw$, and hence $z\in D_{s-1}(u)\cap D_{s-1}(v)$, a
contradiction. Thus $z=z'w$, where necessarily
$z'\in D_s(u)\cap D_s(v)$. This proves the reverse inclusion and
completes the proof.
\end{proof}

\begin{proposition}\label{prop:lower}
If $m$ is $(q,s)$-deletion-feasible, then
\[
\max_{\substack{x,y\in\Sigma_q^n\\
r(x),\,r(y)\le \gamma n\\
|D_s(x)\cap D_s(y)|\le m}}
|D_t(x)\cap D_t(y)|
\ge
\frac{m\gamma^{t-s}}{(t-s)!}n^{t-s}
+
O_{s,t,m}(n^{t-s-1}).
\]
\end{proposition}

\begin{proof}
Let $u,v\in\Sigma_q^L$ be a seed from Definition~\ref{def:feasible}, and write
\[
D_s(u)\cap D_s(v)=\{c_1,\ldots,c_m\}.
\]
Choose any sequence $w\in\Sigma_q^{n-L}$ satisfying $r(w)= \lfloor \gamma n\rfloor-L$, and set
\[
x=uw,
\qquad
y=vw.
\]
By definition, 
\[
  r(x)\leq \gamma n
\qquad
 r(y)\leq \gamma n.
\]
Moreover, by Lemma~\ref{lem:suffix-cancellation},
\[
D_s(x)\cap D_s(y)
=
\{c_1w,\ldots,c_mw\}.
\]
Thus $x$ and $y$ are admissible. 

For every $1\le i\le m$ and every $z\in D_{t-s}(w)$, the sequence
$c_i z$ belongs to $D_t(x)\cap D_t(y)$. Moreover, distinct pairs $(i,z)$ give distinct sequences
$c_i z$. Hence
\[
|D_t(x)\cap D_t(y)|
\ge
m\,|D_{t-s}(w)|.
\]
By Lemma \ref{lem:ballsize},
\[
  |D_{t-s}(w)|\geq \binom{\lfloor \gamma n \rfloor-L-t+s+1}{t-s}= \frac{\gamma^{t-s}}{(t-s)!}n^{t-s}
+
O_{s,t}(n^{t-s-1}).
\]
Therefore
\[
|D_t(x)\cap D_t(y)|
\ge
\frac{m\gamma^{t-s}}{(t-s)!}n^{t-s}
+
O_{s,t,m}(n^{t-s-1}),
\]
which gives the desired lower bound.
\end{proof}

\begin{proof}[Proof of Theorem~\ref{thm:main}]
The upper bound follows from Theorem~\ref{thm:run-upper}, and the lower
bound follows from Proposition~\ref{prop:lower}.
\end{proof}

\section{Proofs of Theorems \ref{thm:finite-insertion-bound} and~\ref{thm:insertion-main}}\label{sec:insertion}

We prove the insertion analogue by following the same primitive-layer
strategy used in Section~\ref{sec:layers}, with common subsequences
replaced by common supersequences. Recall that, for a sequence
$x\in\Sigma_q^n$,
\[
I_t(x)=\{w\in\Sigma_q^{n+t}: x\preccurlyeq w\}.
\]

\begin{definition}\label{def:minimal-common-supersequences}
Let $u$ and $v$ be sequences and let $N\ge0$. A sequence $w$ is a
\emph{minimal common supersequence} of $u$ and $v$ of length $N$ if
$u\preccurlyeq w$, $v\preccurlyeq w$, and there is no common
supersequence $w^-$ of $u$ and $v$ of length $N-1$ such that
$w^-\preccurlyeq w$. Let $\mathcal S_N(u,v)$ be the set of all such sequences.
\end{definition}

\begin{remark}
A minimal common supersequence need not be a shortest common supersequence.
\end{remark}

\begin{lemma}\label{lem:minimal-common-supersequences}
Let $a,b,N$ be nonnegative integers, and let $u,v$ be sequences satisfying
\[
|u|=N-a,
\qquad
|v|=N-b.
\]
Then
\[
|\mathcal S_N(u,v)|\le \binom{a+b}{a}.
\]
\end{lemma}

\begin{proof}
We argue by induction on $|u|+|v|$. The cases $N=0$, $a=0$, or $b=0$
are immediate. Assume $N,a,b\ge1$, and write
\[
u=\alpha u',
\qquad
v=\beta v',
\]
where $\alpha$ and $\beta$ are the first symbols of $u$ and $v$, respectively.

First suppose that $\alpha=\beta$. Then every sequence in
$\mathcal S_N(u,v)$ starts with $\alpha$; otherwise deleting its first
symbol would still leave a common supersequence of $u$ and $v$. Hence
removing this first symbol gives an injection from $\mathcal S_N(u,v)$
into $\mathcal S_{N-1}(u',v')$. By the induction hypothesis,
\[
|\mathcal S_N(u,v)|
\le
|\mathcal S_{N-1}(u',v')|
\le
\binom{a+b}{a}.
\]

It remains to consider the case $\alpha\ne\beta$. A sequence in
$\mathcal S_N(u,v)$ must start with either $\alpha$ or $\beta$, otherwise deleting its first symbol would still leave a common supersequence of $u$ and $v$. If it
starts with $\alpha$, then deleting this first symbol gives an element of
$\mathcal S_{N-1}(u',v)$; if it starts with $\beta$, then deleting this
first symbol gives an element of $\mathcal S_{N-1}(u,v')$. Therefore, by
induction,
\[
|\mathcal S_N(u,v)|
\le
|\mathcal S_{N-1}(u',v)|+|\mathcal S_{N-1}(u,v')|
\le
\binom{a+b-1}{a}
+
\binom{a+b-1}{a-1}
=
\binom{a+b}{a}.
\]
This completes the proof.
\end{proof}

We now introduce the insertion analogue of the primitive-layer decomposition.

\begin{definition}\label{def:insertion-primitive-layer}
For $x,y\in\Sigma_q^n$ and $0\le j\le t$, define
\[
\mathcal B_j(x,y)
:=
I_j(x)\cap I_j(y)
\setminus
\bigcup_{z\in I_{j-1}(x)\cap I_{j-1}(y)} I_1(z),
\]
where $I_{-1}(x)\cap I_{-1}(y)=\varnothing$ by convention.
\end{definition}

\begin{lemma}[Primitive-layer decomposition]\label{lem:insertion-primitive-layer}
For any $x,y\in\Sigma_q^n$ and $0\le j\le t$,
\[
\mathcal B_j(x,y)=\mathcal S_{n+j}(x,y).
\]
Consequently,
\[
|\mathcal B_j(x,y)|\le \binom{2j}{j}.
\]
Moreover, if $1\le s\le t$, then
\[
I_t(x)\cap I_t(y)
=
\left(
\bigcup_{z\in I_s(x)\cap I_s(y)} I_{t-s}(z)
\right)
\cup
\left(
\bigcup_{j=s+1}^{t}\ \bigcup_{z\in\mathcal B_j(x,y)} I_{t-j}(z)
\right).
\]
\end{lemma}

\begin{proof}
A sequence $w\in I_j(x)\cap I_j(y)$ is not in $\mathcal B_j(x,y)$ if and
only if $w\in I_1(z)$ for some $z\in I_{j-1}(x)\cap I_{j-1}(y)$.
Equivalently, $w$ contains a common supersequence of $x$ and $y$ of
length $n+j-1$. Thus $\mathcal B_j(x,y)$ is exactly the set of minimal
common supersequences of $x$ and $y$ of length $n+j$, namely
\[
\mathcal B_j(x,y)=\mathcal S_{n+j}(x,y).
\]
Applying Lemma~\ref{lem:minimal-common-supersequences} with $N=n+j$ and
$a=b=j$ gives
\[
|\mathcal B_j(x,y)|
=
|\mathcal S_{n+j}(x,y)|
\le
\binom{2j}{j}.
\]

It remains to prove the decomposition. 
The inclusion ``$\supseteq$'' is immediate. For the reverse inclusion,
take $w\in I_t(x)\cap I_t(y)$. If $w$ contains some sequence of
$I_s(x)\cap I_s(y)$, then $w$ belongs to the first union. Otherwise, choose the
smallest $j>s$ for which there exists $z\in I_j(x)\cap I_j(y)$ with
$z\preccurlyeq w$. Such a $j$ exists by taking $j=t$ and $z=w$. By the
minimality of $j$, the chosen $z$ cannot be obtained by inserting one
symbol into a sequence of $I_{j-1}(x)\cap I_{j-1}(y)$; hence $z\in\mathcal B_j(x,y)$.
Since $|z|=n+j$ and $|w|=n+t$, we have $w\in I_{t-j}(z)$. Hence $w$ belongs to the second union. This proves the
reverse inclusion and completes the proof.
\end{proof}

\begin{proof}[Proof of Theorem~\ref{thm:finite-insertion-bound}]
Fix $x,y\in\Sigma_q^n$ satisfying $|I_s(x)\cap I_s(y)|\le m$, by the primitive-layer decomposition in Lemma~\ref{lem:insertion-primitive-layer},
\begin{equation}\label{eq:insertion-layer-cover}
I_t(x)\cap I_t(y)
=
\left(
\bigcup_{z\in I_s(x)\cap I_s(y)} I_{t-s}(z)
\right)
\cup
\left(
\bigcup_{j=s+1}^{t}\ \bigcup_{z\in\mathcal B_j(x,y)} I_{t-j}(z)
\right).
\end{equation}
For a sequence $z$ of length $M$, we use the standard formula
\[
|I_h(z)|=\sum_{i=0}^{h}\binom{M+h}{i}(q-1)^i.
\]
Since $|I_s(x)\cap I_s(y)|\leq m$ and each $z\in I_s(x)\cap I_s(y)$ has length $n+s$, the first union in
\eqref{eq:insertion-layer-cover} contributes at most
\[
m\sum_{i=0}^{t-s}\binom{n+t}{i}(q-1)^i= 
\frac{m(q-1)^{t-s}}{(t-s)!}n^{t-s}
+
O_{q,s,t,m}\!\left(n^{t-s-1}\right).
\]
For the second union, Lemma~\ref{lem:insertion-primitive-layer} gives
$|\mathcal B_j(x,y)|\le\binom{2j}{j}$. Hence
\[
\sum_{j=s+1}^{t}\sum_{z\in\mathcal B_j(x,y)} |I_{t-j}(z)|
\le
\sum_{j=s+1}^{t}\binom{2j}{j}
\sum_{i=0}^{t-j}\binom{n+t}{i}(q-1)^i
=
O_{q,s,t}\!\left(n^{t-s-1}\right).
\]
Combining the two estimates yields
\[
|I_t(x)\cap I_t(y)|
\le
\frac{m(q-1)^{t-s}}{(t-s)!}n^{t-s}
+
O_{q,s,t,m}\!\left(n^{t-s-1}\right),
\]
as claimed.
\end{proof}

We next prove the matching lower bound for insertion-feasible values.

\begin{lemma}\label{lem:insertion-suffix-cancellation}
Let $u,v\in\Sigma_q^L$ be such that $I_{s-1}(u)\cap I_{s-1}(v)=\varnothing$.
Then, for every sequence $w$,
\[
I_s(uw)\cap I_s(vw)
=
\{cw:\ c\in I_s(u)\cap I_s(v)\}.
\]
\end{lemma}

\begin{proof}
It suffices to prove the claim when $w$ is a single symbol; the general
case then follows by iteration.

The inclusion ``$\supseteq$'' is immediate. For the reverse inclusion,
take $w\in\Sigma_q$ and $z\in I_s(uw)\cap I_s(vw)$. If $z$ does not end
with $w$, then deleting the last symbol of $z$ and one occurrence used to
realize the terminal $w$ leaves a sequence in $I_{s-1}(u)\cap I_{s-1}(v)$, a contradiction. Hence $z=z'w$, where
necessarily $z'\in I_s(u)\cap I_s(v)$. This proves the reverse inclusion and completes the proof.
\end{proof}

\begin{proposition}\label{prop:insertion-lower}
If $m$ is $(q,s)$-insertion-feasible, then
\[
\max_{\substack{x,y\in\Sigma_q^n\\
|I_s(x)\cap I_s(y)|\le m}}
|I_t(x)\cap I_t(y)|
\ge
\frac{m(q-1)^{t-s}}{(t-s)!}n^{t-s}
+
O_{q,s,t,m}(n^{t-s-1}).
\]
\end{proposition}

\begin{proof}
Let $u,v\in\Sigma_q^L$ be a seed from
Definition~\ref{def:insertion-feasible}, and write
\[
I_s(u)\cap I_s(v)=\{c_1,\ldots,c_m\}.
\]
Choose any sequence $w\in\Sigma_q^{n-L}$ and set
\[
x=uw,
\qquad
y=vw.
\]
By Lemma~\ref{lem:insertion-suffix-cancellation},
\[
I_s(x)\cap I_s(y)=\{c_1w,\ldots,c_mw\}.
\]
Thus $x$ and $y$ are admissible.

For every $1\le i\le m$ and every $z\in I_{t-s}(w)$, the sequence
$c_i z$ belongs to $I_t(x)\cap I_t(y)$. Moreover, distinct pairs $(i,z)$ give distinct sequences
$c_i z$. Hence
\[
|I_t(x)\cap I_t(y)|
\ge
m\,|I_{t-s}(w)|.
\]
Since insertion balls are regular,
\[
|I_{t-s}(w)|
=
\sum_{j=0}^{t-s}\binom{n-L+t-s}{j}(q-1)^j
=
\frac{(q-1)^{t-s}}{(t-s)!}n^{t-s}
+O_{q,t,s,m}(n^{t-s-1}).
\]
Therefore
\[
|I_t(x)\cap I_t(y)|
\ge
\frac{m(q-1)^{t-s}}{(t-s)!}n^{t-s}
+O_{q,t,s,m}(n^{t-s-1}),
\]
which gives the desired lower bound.
\end{proof}

\begin{proof}[Proof of Theorem~\ref{thm:insertion-main}]
The upper bound follows from Theorem~\ref{thm:finite-insertion-bound},
and the lower bound follows from Proposition~\ref{prop:insertion-lower}.
\end{proof}

\section{Concluding Remarks}\label{sec:conclusion}

In this paper, we studied extremal intersections of deletion and
insertion balls. Our first result establishes a finite upper bound for
$|D_t(x)\cap D_t(y)|$ under the condition
$D_{s-1}(x)\cap D_{s-1}(y)=\varnothing$, strengthening a recent result of Pham, Goyal, and Kiah \cite{Pham-25-JCTA} in the binary case. We then introduced maximal common subsequences and a primitive-layer decomposition of common subsequences, and used them to derive asymptotically sharp bounds on
deletion-ball intersections under simultaneous run-count and lower-order
intersection constraints. Finally, we developed a parallel theory for
insertion balls and obtained asymptotically sharp bounds on
insertion-ball intersections under lower-order intersection constraints.
Our results yield direct lifting theorems for deletion and insertion
reconstruction codes, allowing reconstruction properties at radius $s$ to
be transferred to larger radii $t$.

Several open problems remain.

\begin{itemize}
  \item First, it would be desirable to determine the complete sets of
  $(q,s)$-deletion-feasible and $(q,s)$-insertion-feasible integers.
  It is known that $\binom{2s}{s}$ is both $(q,s)$-deletion-feasible \cite{Pham-25-JCTA,Alon-24-IT} and
$(q,s)$-insertion-feasible \cite{Sala-17-IT}. We conjecture that every integer
$m<\binom{2s}{s}$ is also both $(q,s)$-deletion-feasible and
$(q,s)$-insertion-feasible. For $s=1,2$, this conjecture is supported by
simple examples; see Tables~\ref{tab:deletion-feasible-small-s} and \ref{tab:insertion-feasible-small-s}. Developing a uniform construction, or a
general proof, for all such feasible integers would be an interesting
problem.

\begin{table}[t]
\centering
\renewcommand{\arraystretch}{1.15}
\caption{Examples for deletion-feasible integers when $s=1,2$. In each
row, $D_{s-1}(x)\cap D_{s-1}(y)=\varnothing$ and
$|D_s(x)\cap D_s(y)|=m$.}
\label{tab:deletion-feasible-small-s}
\begin{tabular}{c c c c p{0.54\textwidth}}
\hline
$s$ & $m$ & $x$ & $y$ & $D_s(x)\cap D_s(y)$ \\
\hline
$1$ & $1$ & $00$ & $10$
& $\{0\}$ \\

$1$ & $2$ & $01$ & $10$
& $\{0,1\}$ \\
\hline
$2$ & $1$ & $000$ & $011$
& $\{0\}$ \\

$2$ & $2$ & $001$ & $110$
& $\{0,1\}$ \\

$2$ & $3$ & $0011$ & $1010$
& $\{00,01,11\}$ \\

$2$ & $4$ & $0110$ & $1001$
& $\{00,01,10,11\}$ \\

$2$ & $5$ & $010110$ & $100101$
& $\{0010,0011,0101,1010,1011\}$ \\

$2$ & $6$ & $010101$ & $100110$
& $\{0010,0011,0110,1001,1010,1011\}$ \\
\hline
\end{tabular}
\end{table}

\begin{table}[t]
\centering
\renewcommand{\arraystretch}{1.15}
\caption{Examples for insertion-feasible integers when $s=1,2$. In each
row, $I_{s-1}(x)\cap I_{s-1}(y)=\varnothing$ and
$|I_s(x)\cap I_s(y)|=m$.}
\label{tab:insertion-feasible-small-s}
\begin{tabular}{c c c c p{0.54\textwidth}}
\hline
$s$ & $m$ & $x$ & $y$ & $I_s(x)\cap I_s(y)$ \\
\hline
$1$ & $1$ & $001$ & $100$
& $\{1001\}$ \\

$1$ & $2$ & $0$ & $1$
& $\{01,10\}$ \\
\hline
$2$ & $1$ & $00011$ & $10100$
& $\{1010011\}$ \\

$2$ & $2$ & $0011$ & $1100$
& $\{001100,110011\}$ \\

$2$ & $3$ & $0001$ & $1100$
& $\{011001,101001,110001\}$ \\

$2$ & $4$ & $00100$ & $10011$
& $\{0010011,1001001,1001010,1001100\}$ \\

$2$ & $5$ & $00010$ & $10011$
& $\{0100101,0100110,1000101,1000110,1001010\}$ \\

$2$ & $6$ & $00$ & $11$
& $\{0011,0101,0110,1001,1010,1100\}$ \\
\hline
\end{tabular}
\end{table}

  \item Second, the lower-order terms in Theorems~\ref{thm:main} and~\ref{thm:insertion-main} remain poorly understood. Determining explicit second-order terms, or more generally obtaining complete expansions for the corresponding extremal intersection sizes, is an important direction for future work.

  \item Finally, it would be interesting to study constructions
for $(n,q,N;D_s)$-reconstruction codes and
$(n,q,N;I_s)$-reconstruction codes in the regime
$N\le \binom{2s}{s}$. By the lifting theorems proved in this paper,
any such results would immediately yield reconstruction codes for larger
deletion or insertion radii. Existing works
\cite{Sun-26-IT,Sun-23-IT,Sun-25-IT-C,Chrisnata-22-IT,Cai-22-IT,Chee-18-IT}
have mainly focused on the cases $s\le 2$.
For the case of three deletions, Zhang, Ge, and Zhang
\cite{Zhang-24-ISIT} took an initial step by characterizing the structure
of pairs of sequences whose $3$-deletion balls have intersection size
$19$ or $20$, under the assumption that their $2$-deletion balls are
disjoint. Extending the
theory to the general case $s\ge 3$ remains an interesting direction for
future research.
\end{itemize}

\bibliographystyle{plain}
\bibliography{ref}

\end{document}